\documentclass[reqno,12pt]{amsart}
\usepackage{amscd}
\usepackage{amssymb}
\usepackage[a4paper]{geometry}
\usepackage{amstext}
\usepackage{mathrsfs}
\usepackage{times} 
\usepackage{todonotes}
\usepackage[final]{hyperref}
\hypersetup{unicode= false, colorlinks=true, linkcolor=blue,
anchorcolor=blue, citecolor=red, filecolor=red, menucolor=blue,
pagecolor=blue, urlcolor=blue}
\numberwithin{equation}{section}
\usepackage[none]{hyphenat}

\newcommand{\CC}{\mathbb {C}}

\newtheorem{thm}{Theorem}
\newtheorem{proposition}{Proposition}
\newtheorem{lemma}{Lemma}
\newtheorem{corollary}{Corollary}

\title[Product of Volterra type integral  and composition  operators on weighted Fock spaces]
 {Product of Volterra type integral  and composition operators on weighted Fock spaces}
\author [Tesfa  Mengestie]{Tesfa  Mengestie }
\address{Department of Mathematical Sciences\\
Norwegian University of Science and Technology (NTNU)\\
 NO- 7491 Trondheim, Norway}
\email{mengesti@math.ntnu.no /mengestie@yahoo.com}
\subjclass[2000]{30E05, 46E22}
 
 \keywords{Fock space, Volterra operator, composition operator, boundedness, compactness, Schatten class, norm, essential norm}
\begin{document}
\begin{abstract}
We characterize the  bounded, compact, and Schatten class product of
Volterra type integral and composition operators acting between
weighted Fock spaces. Our results are expressed in terms of certain
Berezin type integral transforms on the complex plane $\CC$. We also
estimate the  norms and  essential norms of these operators in terms
of the integral transforms. All our results are valid for weighted
composition operators when acting between the class of weighted Fock
spaces considered.
\end{abstract}
 \maketitle

\section{Introduction}
Given a space of functions $\mathscr{H}$ of holomorphic functions on
$\CC,$ we define the Volterra type integral operator on
$\mathscr{H}$ induced by a holomorphic symbol $g$ by
 \begin{equation*}
 V_gf(z)= \int_0^z f(w)g'(w) dw.
\end{equation*}
Questions about boundedness, compactness, and other  operator
theoretic properties of $V_g$ expressed in terms of function
theoretic conditions on $g$ have been a subject of high interest
since introduced by Pommerenke \cite{Pom} in 1997. The operator
$V_g$ has  in particular attracted a lot of interest following the
works of Aleman and Siskakis \cite{Alsi1,Alsi2} on Hardy and
Bergmann spaces.  For more information,  we refer to the surveys in
\cite{Alman,Si} and the references therein. The Volterra type
integral operator $V_g$ has an interesting  relation with the
multiplication operator $Mg(f)= gf$ by $M_g(f)= f(0)g(0)+ V_g(f)+
I_g(f),$ where $I_g$ is the Volterra companion integral operator
given by
\begin{equation}
\label{campanian} I_gf(z)= \int_0^z f'(w)g(w)dw.
\end{equation}

 Let $\psi$ be entire function and  $C_\psi f= f(\psi)$ be
the induced composition operator on  space of analytic functions on
$\CC.$ We define the product of Volterra type integral and
composition operators induced by the  pair of symbols $(g,\psi)$ by
\begin{equation}
\label{define} V_{(g, \psi)}f(z)= \int_0^z f(\psi(w))g'(w)dw.\
\end{equation}
If $\psi(z)=z,$ then these operators are just  the usual Volterra
type integral operators $V_g.$ As will be seen latter the study of
$V_{(g,\psi)}$   reduce to studying the  composition operator
$C_\psi$ when $|g'(z)/(1+|z|)|$ behaves like a constant for all $z$.
Several authors have studied operators of this kind\cite{LIS, SLI,
AS, XZHU,EW}. $\breve{\text{Z}}$.
 $\breve{\text{C}}$u$\breve{\text{c}}$kovi$\acute{\text{c}}$ and R.
Zhao \cite{ZZHH, ZZH} characterized the bounded and compact weighted
composition operators between different weighted Bergman spaces and
different Hardy spaces in terms of the generalized Berezin
transform. Similar results were also obtained in  \cite{UEKI}  for
the same operator acting on the classical Fock space $F_1^2$.

 In this paper, we present analogous  results for product of Volterra
type integral and composition operator $V_{(g,\psi)}$ when it acts
between different weighted Fock spaces. By modifying all the results
stated for $V_{(g,\psi)},$ one could also obtain similar results for
the  operators
\begin{equation*}
 C_{(\psi,g)} f(z):= \int_{0}^z f(\psi(w)) (g(\psi(w))'dm(w).
\end{equation*}

 We recall that for $0<p<\infty $  and $\alpha >0, $ the
Fock space $F_\alpha^p$ consists of entire functions $f$ for which
\begin{equation*}
\|f\|_{(p,\alpha)}^p=  \frac{p\alpha}{2\pi}\int_{\CC} |f(z)|^p
e^{-\frac{p \alpha}{2}|z|^2} dm(z) <\infty
\end{equation*} where $dm$ is the Lebesgue measure.  In particular, $F_\alpha^2$ is a
reproducing kernel Hilbert space with kernel and normalized kernel
functions respectively $K_{(w,\alpha)}(z)= e^{\alpha \overline{w}z}$
and $k_{(w,\alpha)}(z)= e^{-\alpha|w|^2/2+ \alpha\overline{w}z}.$

 Our results will be expressed in terms of the Berezin type integral transform
\begin{eqnarray*}
B_{(\psi, \alpha)}(|g|^p)(w)&=& \int_{\CC} e^{\frac{p\alpha}{2}\big(
2\Re \langle \psi(z), \
w\rangle-|z|^2-|w|^2\big)}\frac{|g'(z)|^p}{(1+|z|)^p} dm(z),
\end{eqnarray*} where $\langle.,.\rangle$ is the standard inner product in
the complex plan $\CC$.
\subsection{bounded and compact $V_{(g,\psi)}$} We now state our  first main
result.
\begin{thm}\label{bounded}
Let $0<p\leq q<\infty$ and $\psi$ be an entire function. Then
$V_{(g,\psi)}: F_\alpha^p \to F_\alpha^q$ is \begin{enumerate}
 \item bounded if and only if
$ B_{(\psi, \alpha)}(|g|^q) \in L^\infty(\CC, dm)$.
Moreover\footnote{The notation $U(z)\lesssim V(z)$ (or equivalently
$V(z)\gtrsim U(z)$) means that there is a constant $C$ such that
$U(z)\leq CV(z)$ holds for all $z$ in the set in question, which may
be a Hilbert space or  a set of complex numbers. We write
$U(z)\simeq V(z)$ if both $U(z)\lesssim V(z)$ and $V(z)\lesssim
U(z)$.}
\begin{equation}
\label{norm1estimate} \|V_{(g,\psi)}\| \simeq \Big( \sup_{w\in
\CC}B_{(\psi, \alpha)}(|g|^q)(w)\Big)^{1/q}.
\end{equation}
\item compact if and only if $\lim_{|z|\to \infty}
B_{(\psi,\alpha)}(|g|^q) (z)= 0$.
\end{enumerate}
\end{thm}
The conditions both in (i) and (ii) are independent of the
 exponent $p$ apart from the fact that $p$ should not be exceeding
 $q$. It means that if there exists a $p>0$ for which  $ V_{(g,
\psi)}$ is bounded (compact) from $ F_\alpha^p $ to $ F_\alpha^q$,
then it is also  bounded (compact) for every other $p \leq q$.

 It is often difficult to
determine whether a concrete operator on a function space possesses
properties such as boundedness, compactness or Schatten class
membership. For reproducing kernel Hilbert spaces, a fruitful
strategy has been to test whether the operator theoretic properties
could  be determined by its action on the kernel functions alone. In
general there is no reason why this should hold but many important
results are known to be interpreted as examples of this property
which we call the reproducing kernel thesis property. Our results
above present another example of the thesis property.  The
boundedness and compactness of $V_{(g, \psi)}$ are respectively
equivalent to
\begin{equation*}\sup_{z\in \CC} \|V_{(g, \psi)}k_{(z,\alpha)}\|_{(q,\alpha)}<\infty \
\text{and}\ \lim_{|z|\to \infty}
 \|V_{(g, \psi)}k_{(z,\alpha)}\|_{(q,\alpha)}=0.
 \end{equation*}

A natural question is whether there exists an interplay between the
two symbols  $g$ and $\psi$  in inducing bounded and compact
operators $V_{(g,\psi)}$. We first observe that if $g'\neq 0,$ then
by the classical
 Liouville's theorem  the function $g$ can not decay in
any way. This forces that
\begin{eqnarray*}
B_{(\psi, \alpha)}(|g|^p)(w) =\int_{\CC}e^{\frac{p\alpha}{2}\big(
|\psi(z)|^2-|z|^2 -|\psi(z)-w|^2\big)}\frac{|g'(z)|^p}{(1+|z|)^p}
dm(z)
\end{eqnarray*} is   bounded    only when  $\psi(z)= az+b$ with
$|a|\leq 1$. Moreover if $|a|=1,$ then $b=0, $ and  compactness is
achieved only when $|a|<1.$ Combining this with  Proposition 3 in
\cite{CMS}, we get the following\footnote{The  Fock type spaces in
\cite{CMS} are defined in a slightly different way than ours. }.
\begin{corollary} Let $0<p\leq q<\infty,$ $g' \neq 0$ and $\psi$ be an entire function.
Then if $V_g\circ C_\psi= V_{(g,\psi)}: F_\alpha^p \to F_\alpha^q$
is
\begin{enumerate}
 \item  bounded, then  $C_\psi$ is bounded.
 \item  compact, then $C_\psi$ is compact.
\end{enumerate}
\end{corollary}
On the other hand a bounded  $V_{(g,\psi)}$ does not necessarily
imply boundedness of the Volterra type integral operator  $V_g$.
This is because boundedness of the former  allows  $g$ to be any
entire function that grows more slowly than the exponential part of
the integrand in $B_{(\psi, \alpha)}(|g|^p)(w)$ while boundedness of
the latter forces $g$ to grow as a power function of at most degree
2, as seen below.  By setting $\psi(z)=z$ in the theorem, we
immediately get the following
 result of Constantin \cite{Olivia}.
 \begin{corollary}\label{cor1}
 Let $0<p\leq q<\infty$. Then $V_{g}: F_\alpha^p \to F_\alpha^q$ is \begin{enumerate}
 \item bounded if and only if $g(z)= az^2+bz+c,\ a, b, c\in \CC$.
\item compact if and only if $g(z)= az+b$.
\end{enumerate}
 \end{corollary}
\begin{proof}
  By Theorem~\ref{bounded}, the
sufficiency of the conditions both in (i) and (ii) are immediate. We
shall sketch the necessity. If $D(w,1)=\{z\in \CC: |z-w|< 1\},$ then
\begin{equation}\label{voltnecess}
 B_{(\psi,\alpha)}(|g|^q)(w) \gtrsim \int_{
D(w,1)}\frac{|g'(z)|^q}{(1+|z|)^q} dm(z) \gtrsim
\frac{|g'(w)|^q}{(1+|w|)^q},
\end{equation} where the last inequality follows by subharmonicity.
 Assuming boundedness of $V_{(g,\psi)}$, \eqref{voltnecess} implies
$|g'(w)|\lesssim 1+|w|$ for all $w \in \CC,$ from which the desired
expression for $g$ follows. On the other hand, if $V_g$ is compact,
then  since $k_{(w,\alpha)} \to 0$ uniformly on compact subsets of
$\CC$ as $|w| \to \infty$ we see from relation \eqref{voltnecess}
that
\begin{equation*}\frac{|g'(w)|}{1+|w|} \to 0, \ \text{as} \ |w| \to
\infty.\end{equation*} This can happen only when $g$ is  a
polynomial of degree at most 1.\end{proof}
 Interestingly, many more $g's$ are
admissible  than those in the previous corollary if we  scale $\psi$
as $\psi(z)= \beta z$ with $|\beta|<1.$ More precisely, we get the
following whose proof is just immediate from the theorem.
\begin{corollary} Let
$\psi(z)= \beta z$ with $|\beta|<1$ and  $0<p\leq q<\infty$. Then
$V_{(g,\psi)}: F_\alpha^p \to F_\alpha^q$ is bounded for any $g$
such that
\begin{equation*}
|g(z)| \lesssim e^{\frac{\alpha \gamma}{2}|z|^2}
\end{equation*}for all  $z$ in $\CC$ and any $\gamma$ satisfying $\gamma+|\beta|^2<1.$
\end{corollary}
Theorem~\ref{bounded} and all the subsequent results are valid for
the weighted
 composition operator $(u C_\psi) f(z)= u(z).f(\psi(z)) $ between the
Fock spaces as described  here where $u$ is an entire function on
$\CC.$  We only  have to replace
 the weight  $|g'(z)|/(1+|z|)$ by  $ |u(z)|$ to get the corresponding
 results. For $p=q=2$ and $\alpha=1$, the bounded and compact
 composition operators were described in \cite{UEKI} apart from missing the
 fact that $\psi$ can be nothing but linears.

 For the case where we map  larger weighted Fock spaces into smaller ones, we get the following
 stronger conditions as one would expect.
\begin{thm} \label{thsmall} Let $0<q<p<\infty$ and $\psi$ be an entire function. Then
the following are equivalent
\begin{enumerate}
 \item $V_{(g,\psi)}: F_\alpha^p \to F_\alpha^q$ is bounded.
  \item $V_{(g,\psi)}: F_\alpha^p \to F_\alpha^q$ is  compact.
  \item  $B_{(\psi,\alpha)}(|g|^q)\in L^{\frac{p}{p-q}}(\CC, dm)$. Moreover,
\begin{equation}
\label{normless} \|V_{(g,\psi)}\| \simeq \Big(\int_{\CC}
B^{p/(p-q)}_{(\psi,\alpha)}(|g|^q)(w)dm(w)\Big)^{(p-q)/p}
\end{equation}
\end{enumerate}
\end{thm}
It is interesting to  note that unlike condition (iii) of
Theorem~\ref{bounded} where we map smaller spaces into bigger ones,
condition (iii) above is expressed in terms of both exponents $p$
and $q$. When $\psi(z)=z,$ then the theorem simplifies to saying
that $V_g$ (for non constant $g$) is bounded or compact if and only
if $g'$ is a constant $q>2p/(p+2)$ and $g'=0$ for $q<2p/(p+2)$. This
is because by subharmonicity,
\begin{eqnarray}
\label{bddcomp}
 \int_{\CC}
B^{\frac{p}{p-q}}_{(\psi,\alpha)}(|g|^q)(w)dm(w) &\gtrsim&
\int_{\CC}\bigg( \int_{D(w,1)}
\bigg|\frac{k_{(w,\alpha)}(z)g'(z)}{(1+|z|)}\bigg|^q e^{-\frac{q\alpha}{2}|z|^2}dm(z)\bigg)^{\frac{p}{p-q}}dm(w)\nonumber\\
&\geq& \int_{\CC}\bigg|\frac{g'(w)}{1+|w|}\bigg|^{1/q} dm(w)
\end{eqnarray} from which the desired restrictions on $g,$  $p$ and
$q$ follow once we assume that the left-hand side of \eqref{bddcomp} is finite.\\
Observe that by setting $|g'(z)/(1+|z|)| \simeq 1,$
Theorem~\ref{thsmall} characterizes the bounded and compact
composition operators from $F_\alpha^p$ to $F_\alpha^q$ whenever
$p>q.$ This extends the result in \cite{CMS} where similar
conditions are given for compact and bounded $C_\psi:F_\alpha^p \to
 F_\alpha^q$ whenever  $0<p\leq q<\infty$. Those conditions in \cite{CMS} for the one variable setting
 could also  be obtained easily from Theorem~\ref{bounded}.  The
corresponding conditions in  Theorem~\ref{thsmall} can be easily
simplified to give that $C_\psi: F_\alpha^p \to F_\alpha^q$ is
bounded (compact) if and only if $\psi(z)= az+b$ with $|a|<1.$
 \subsection{Essential norm of $V_{(g,\psi)}$}The essential norm
$\|T\|_e$ of a bounded operator $T$ on a Banach space $\mathscr{H}$
is defined as the distance from $T$ to the space of compact
operators on $\mathscr{H}.$ We refer to \cite{ZZHH,ZZH,Shapiro,
UEKI,DV} for estimation of such norms  for  different operators on
Hardy space, Bergman space, $L^p$ and some Fock spaces. We get the
following estimate for $V_{(g,\psi)}.$
\begin{thm}\label{essentialnorm}
Let $1 <p\leq q< \infty$ and $\psi$ be an entire function. If
$V_{(g,\psi)}:F_\alpha^p \to F_\alpha^q$ is bounded,  then
\begin{equation}
\label{essential} \|V_{(g,\psi)}\|_e\simeq \bigg(\limsup_{|w| \to
\infty} B_{(\psi,\alpha)}(|g|^q)(w) \bigg)^{1/q}.
\end{equation}
\end{thm}
For $p>1,$ the compactness condition in Theorem~\ref{bounded} could
be easily drawn from this relation  since the left-hand side
expression \eqref{essential} vanishes for compact $V_{(g,\psi)}.$ In
particular when $\psi(z)= z,$  a simple computation
 along with Corollary~\ref{cor1}  shows that
\begin{equation*}
 \| V_{g}\|_e \simeq \sqrt[q]{q}.\end{equation*}
 \subsection{Schatten Class $V_{(g,\psi)}$} Let us now turn to the Schatten class
 membership of $V_{(g,\psi)}$. We recall that a positive
operator $T$ on $F_\alpha^2$ belongs to the trace class if
\begin{equation*}\sum_{n=1}^\infty \langle Te_n, e_n\rangle <\infty \end{equation*} for some
orthonormal basis $(e_n)$ of $F_\alpha^2.$ If $0<p<\infty,$ a
bounded operator $T$  on $F_\alpha^2$  belongs to the Schatten class
$S_p$ if the positive operator  $(T^*T)^{p/2}$ is in the trace
class. We denote the $S_p$ norm of $T$ by $\|T\|_{S_p}.$
\begin{proposition} \label{prop}
Let $\mathscr{H}$ be any Hilbert space and $T$ be a bounded operator
from $F_\alpha^2$ to $\mathscr{H}.$ (i) If $p\geq 2$ and $T\in S_p,$
then
\begin{equation}
\label{necessary} \int_{\CC} \|Tk_{(z,\alpha)}\|_{\mathscr{H}}^p
 dm(z) <\infty.
\end{equation}
(ii)  If $0<p\leq 2$ and \eqref{necessary} holds, then $T\in S_p.$
\end{proposition}
It is  shown in \cite{HOLDW} that the converse to the two statements
above  fail to hold for Hankel operators on the Hardy space  $H^2$.
The interest is now whether the converse still hold for the product
of Volterra type integral and composition  operators under
consideration. It turns out that this is indeed the case (see,
Theorem~\ref{Schatten}).  In particular, $T$ belongs to the
Hilbert--Schmidt class if and only if for any orthonormal basis
$(e_n)$ in $\mathscr{H}$,
\begin{eqnarray}
\label{Hilbert-Schmidt} \|T\|_{S_2}^2&=& \sum_{n=1}^\infty
\|T^*e_n\|_{(2,\alpha)}^2\simeq\sum_{n=1}^\infty \int_{\CC}
T^*e_n(z)\overline{T^*e_n(z)}e^{-\alpha|z|^2}dm(z)\nonumber\\
&=&  \int_{\CC} \sum_{n=1}^\infty |\langle
TK_{(z,\alpha)},e_n\rangle|^2 e^{-\alpha|z|^2}dm(z)\\
&=&\int_{\CC}\|Tk_{(z,\alpha)}\|_{\mathscr{H}}^2dm(z)<\infty.\nonumber
\end{eqnarray}
If $T$ is any positive operator in the trace class of $ F_\alpha^2,$
then by the above
\begin{eqnarray*}
tr(T)= \|T^{1/2}\|_{S_2}^2\simeq \int_{\CC}
\|T^{1/2}k_{(z,\alpha)}\|_{(2,\alpha)}^2 dm(z)=\int_{\CC} \langle
Tk_{(z,\alpha)}, k_{(z,\alpha)}\rangle dm(z).
\end{eqnarray*}
We recall that  $T$ belongs to the Schatten class $S_p$ if and only
if $(T^*T)^{p/2}$ belongs to the trace class. Thus
\begin{eqnarray*}
tr((T^*T)^{p/2})&\simeq&\int_{\CC} \langle
(T^*T)^{p/2}k_{(z,\alpha)}, k_{(z,\alpha)}\rangle dm(z)\gtrsim
\int_{\CC}\|Tk_{(z,\alpha)}\|_{\mathscr{H}}^p dm(z),
\end{eqnarray*} for $p\geq
2$ and the inequality is reversed for $0<p\leq 2.$ In particular
when $T= V_{(g, \psi)},$ we have
\begin{equation*}\label{necess}
\int_{\CC}\|V_{(g,\psi)}k_{(z,\alpha)}\|_{(2,\alpha)}^p dm(z)\simeq
\int_{\CC} B^{p/2}_{(\psi,\alpha)}(|g|^2)(w)dm(w),
\end{equation*} which gives the proofs of the necessity for $p\geq 2$ and
the sufficiency for $0<p\leq 2$ of our next theorem.
\begin{thm}\label{Schatten}
Let $0<p<\infty$ and $\psi$ be an entire function. Then a bounded
map $V_{(g,\psi)}: F_\alpha^2\to F_\alpha^2$ belongs to $S_p$ if and
only if $B_{(\psi,\alpha)}(|g|^2) \in L^{p/2}(\CC,dm)$.
\end{thm}
In particular for  the Volterra type integral operator  $V_g,$ we
obtain the following.
\begin{corollary}
\label{cor3} Let $V_g$ be a compact operator on $F_\alpha^2$.
 If $0<p\leq 2,$ then   $V_g$ belongs to $S_p$ if and only if  $
g$ is a constant function. On the other hand,  $V_g$ belongs to
$S_p$ for all $p> 2.$
\end{corollary}
\begin{proof}
 For $p\geq 2,$ this result was also  proved in \cite{Olivia}. Now we
observe that it in fact follows immediately from
Theorem~\ref{Schatten}. For (i), we observe that it is suffices to
show that there are no nontrivial Hilbert--Schmidt Volterra type
integral operators. The rest will follow from the monotonicity
property of the Schatten classes. To this end, if
 $V_g$ is a compact operator, then by
Corollary~\ref{cor1}, $g'= C,$  a constant. On the other hand by
subharmonicity
\begin{eqnarray*}
\int_{\CC} B^{p/2}_{(\psi,\alpha)}(|g|^2)(w) dm(w)&\simeq&
\int_{\CC}\bigg(\int_{\CC} \frac{C^2
e^{-|z-w|^2}}{(1+|z|)^2}dm(z)\bigg)^{p/2}dm(w)\\
&\gtrsim& \int_{\CC} \frac{C^P}{(1+|w|)^p}dm(w).\end{eqnarray*}
Theorem~\ref{Schatten} ensures that if $V_g$ belongs to $S_p,$ then
the above integrals should converge. But for  $p=2$, this holds only
when $C=0.$ The integral converges for all $p>2.$\end{proof}  By
combining Theorem~\ref{Schatten} with Funbini's Theorem, it is also
easily seen that $ V_{(g,\psi)}$ is a Hilbert--Schmidt operator if
and only if
\begin{equation*}
\int_{\CC} \frac{|g'(z)|^2}{(1+|z|)^2}
e^{\alpha|\psi(z)|^2-\alpha|z|^2} dm(z)<\infty.
\end{equation*}
 As remarked earlier, all our results are valid for weighted composition
 operator $uC_\psi.$  For such operators,
 Theorem~\ref{Schatten} can be simplified further.
 \begin{corollary}\label{corscahtten}
Let $0<p<\infty$ and $\psi$ and  $u$ be entire functions. Then a
bounded map $uC_\psi: F_\alpha^2 \to F_\alpha^2$ belongs to $S_p$ if
and only if $\psi(z)= az+b$ and
 \begin{equation*}
 \int_{\CC} |u(z)|^p e^{\frac{p\alpha}{2}\big((|a|^2-1)|z|^2+2\Re\langle az,b\rangle\big)} dm(z) <\infty
 \end{equation*} for some $a, b\in \CC$ and $|a|<1$.
 \end{corollary}
\emph{Proof.} By Theorem~\ref{Schatten}, $ uC_\psi \in S_p$ if and
 only if
 \begin{eqnarray*}
\int_{\CC} \bigg(\int_{\CC} |u(z)|^2| k_{(w,\alpha)} (\psi(z))|^2
e^{-\alpha|z|^2}dm(z)\bigg)^{p/2} dm(w)\\
\simeq \int_{\CC} \|uC_\psi k_{(w,\alpha)}\|_{(2,\alpha)}^p dm(w)
<\infty.
 \end{eqnarray*}
 On the other hand, $ uC_\psi \in S_p$ if and only if $(uC_\psi)^* \in
 S_p,$ and $\|uC_\psi\|_{S_p} = \|(uC_\psi)^*\|_{S_p} $. From this it follows that
\begin{equation}
\label{norm} \int_{\CC} \|uC_\psi k_{(w,\alpha)}\|_{(2,\alpha)}^p
dm(w)= \int_{\CC} \|(uC_\psi)^* k_{(w,\alpha)}\|_{(2,\alpha)}^p
dm(w).
\end{equation} Since $(uC_\psi)^* k_{(w,\alpha)}= \overline{u(w)}
e^{-\alpha|w|^2/2} K_{(\psi(w),\alpha)},$ we find
\begin{equation*}\|(uC_\psi)^* k_{(w,\alpha)}\|_{(2,\alpha)}= |u(w)| e^{\frac{\alpha}{2}\big(|\psi(w)|^2-|w|^2\big)},
\end{equation*} and
plugging this into \eqref{norm} gives that $uC_\psi\in S_p$ if and
only if \begin{equation*} \int_{\CC} |u(z)|^p
e^{\frac{p\alpha}{2}\big(|\psi(z)|^2-|z|^2\big)} dm(z) <\infty.
\end{equation*} Compactness forces that $\psi(z)= az+b, \ |a|<1$ and
hence the  desired result follows.

 Note that the above argument
can not be  carried  over  in general
to simplify Theorem~\ref{Schatten}. 
 Combining Corollary~\ref{corscahtten} with
Theorem~\ref{bounded}
 immediately gives the following known Schatten class membership criteria
for the composition operator.
\begin{corollary}Let $0<p<\infty$ and $\psi$ be an entire function.
Then the following are equivalent for a bounded map $C_\psi$ .
\begin{enumerate}
\item The map $C_\psi: {F_\alpha^2} \to {F_\alpha^2}$ is compact.
\item The map $C_\psi: {F_\alpha^2}\to {F_\alpha^2}$
belongs to $S_p$ for all $p >0$.
\item $\psi(z)= az+b, a, b\in \CC,$ and $ |a|<1.$
\end{enumerate}
\end{corollary}
\section{Proof of the main results}
 One of the
main tools in proving  our results is the following lemma.
\begin{lemma}
\label{lemdfinite}  Let $f$  be an entire function  and
$0<p<\infty.$ Then
 \begin{equation*}
 \int_{\CC} |f(z)|^p e^{\frac{-\alpha p}{2}|z|^2} dm(z) \simeq |f(0)|^p +
 \int_{\CC} \frac{|f'(z)|^p }{(1+|z|)^p }e^{\frac{-\alpha p}{2}|z|^2}
 dm(z).
\end{equation*}
\end{lemma}The lemma was proved by
Constantin \cite{Olivia} and  describes Fock spaces in terms of
derivatives  as analogous to the case of Bergman spaces in
\cite{PP}. The following estimate will be needed frequently in our
consideration later.
\begin{lemma}
\label{pointwise}  For each $p>0,$ let $\mu_{(p,\alpha)}$ be
positive pull back  measure on $\CC$ defined by \begin{equation*}
\mu_{(p,\alpha)}(E)= \frac{\alpha p}{2\pi} \int_{\psi^{-1}(E)}
\frac{|g'(z)|^p}{(1+|z|)^p} e^{-\frac{\alpha p}{2}|z|^2}
dm(z)\end{equation*} for every Borel  subset $E$ of $\CC$. Then
\begin{equation*}
\int_{D(w,1)}  e^{\frac{\alpha p}{2}|z|^2} d\mu_{(p,\alpha)}(z)
\lesssim e^{\frac{p\alpha}{2}} B_{(\psi,\alpha)}(|g|^p)(w).
\end{equation*}
\end{lemma}
\emph{Proof.} For $p= 2$, the lemma was proved in \cite{UEKI}.  A
modification of that proof  works for other p's which we sketch it
now. For each $z \in D(w, 1),$ observe that
\begin{equation*}|k_{(w,\alpha)}(z)|^p= \big|e^{\alpha\bar{w}z-\frac{\alpha}{2}|w|^2}\big|^p =
e^{\frac{\alpha p}{2} \big( |z|^2-|z-w|^2\big)} \geq
e^{-\frac{\alpha p}{2}+ \frac{p\alpha}{2}|z|^2}.\end{equation*} This
implies
\begin{eqnarray*}e^{-\frac{\alpha p}{2}}\int_{D(w,1)}e^{\frac{p\alpha }{2}|z|^2}
d\mu_{(p,\alpha)}(z)&\leq& \int_{D(w,1)} |k_{(w,\alpha)}(z)|^p
d\mu_{(p,\alpha)}(z)\\
&\leq& \int_{\CC} |k_{(w,\alpha)}(z)|^p d\mu_{(p,\alpha)}(z).
\end{eqnarray*} Invoking the definition of the measure $\mu_{(p,\alpha)}$ and
the integral transform $B_{(\psi,\alpha)}(|g|^p)$ give
 \begin{eqnarray*} \int_{\CC}|k_{(w,\alpha)}(z)|^p d\mu_{(p,\alpha)}(z)
 &\simeq& \int_{\CC} |k_{(w,\alpha)}(\psi(z))|^p
\frac{|g'(z)|^p}{(1+|z|)^p} e^{-\frac{p\alpha}{2}|z|^2} dm(z)\\
&=& B_{(\psi,\alpha)}(|g|^p)(w).  \end{eqnarray*}
 \emph{Proof of Theorem~\ref{bounded}.} (i)
Suppose that $V_{(g, \psi)}$ is bounded. Then a simple computation
shows that $\|k_{w,\alpha}\|_{(p,\alpha)}=1$ for all $p>0.$ Thus
applying $ V_{(g,\psi)} $ on the normalized kernel functions along
with Lemma \ref{lemdfinite} yields
\begin{eqnarray}
\label{normnec} 1&\gtrsim&
\|V_{(g,\psi)}k_{(w,\alpha)}\|_{(q,\alpha)}^q \simeq
B_{(\psi,\alpha)}(|g|^q)(w),
\end{eqnarray} from which the necessity follows.  To  prove the sufficiency we
extend the techniques used in \cite{UEKI}. By definition of the
measure $\mu_{(q,\alpha)},$ Lemma \ref{lemdfinite}, and Lemma~1 in
\cite{JIKZ}
\begin{eqnarray*}
\|V_{(g,\psi)}f\|_{(q,\alpha)}^q&\simeq& \int_{\CC} |f(z)|^q
d\mu_{(q,\alpha)}(z)\\
&\lesssim& \int_{\CC}e^{\frac{\alpha q}{2}|z|^2}
d\mu_{(q,\alpha)}(z) \int_{\CC}\chi_{D(z,1)}(w) |f(w)|^q
e^{-\frac{q\alpha}{2}|w|^2}dm(w)
\end{eqnarray*} where $ \chi_{D(z,1)}$ is the characteristic function
of $D(z,1).$ By Lemma \ref{pointwise}, Fubini's Theorem and the fact
that  $ \chi_{D(z,1)}(w)= \chi_{D(w,1)}(z),$ for all $z, w \in \CC,$
we get
\begin{eqnarray}
\label{normsuff} \|V_{(g,\psi)}f\|_{(q,\alpha)}^q &\lesssim&
\int_{\CC}e^{\frac{q\alpha}{2}|z|^2} d\mu_{(q,\alpha)}(z)
\int_{\CC}\chi_{D(z,1)}(w) |f(w)|^q e^{-\frac{q\alpha}{2}|w|^2}dm(w)\nonumber\\
&=&\int_{\CC}|f(w)|^q e^{-\frac{q\alpha}{2}|w|^2}dm(w)
\bigg(\int_{\CC}\chi_{D(w,1)}(z)e^{\frac{q\alpha}{2}|z|^2}d\mu_{(q,\alpha)}(z)\bigg)dm(w)\nonumber\\
&\lesssim& \int_{\CC} |f(w)|^q e^{-\frac{q\alpha}{2}|w|^2}
B_{(\psi,\alpha)}(|g|^q)(w)
dm(w)\label{normgen}\\
& \lesssim & \sup_{w\in \CC} B_{(\psi,\alpha)}(|g|^q)(w)
\|f\|_{(q,\alpha)}^q\nonumber\\
& \leq& \sup_{w\in \CC} B_{(\psi,\alpha)}(|g|^q)(w)
\|f\|_{(p,\alpha)}^q \label{normsuff}
\end{eqnarray} where the last inequality follows by the inclusion
$F_\alpha^p \subseteq F_\alpha^q.$  From \eqref{normnec} and
\eqref{normsuff}, we
deduce that \eqref{norm1estimate} holds.\\
For (ii), observe that $k_{(w,\alpha)}(z)=
e^{\alpha\overline{w}z-\alpha\frac{1}{2}|w|^2} \to 0$ uniformly on
compact subsets of $\CC$ as $|w|\to \infty.$ It follows that
 \begin{eqnarray*}
0= \lim_{|w|\to \infty}\|V_{(g,\psi)}
k_{(w,\alpha)}\|_{(q,\alpha)}^q &\simeq&\lim_{|w|\to
\infty}B_{(\psi,\alpha)}(|g|^q)(w)
\end{eqnarray*} from which the necessity follows again.  So we remain to show the
sufficiency of the condition.  To this end, we let $f_n$ be a
sequence of entire functions such that $\sup_n \|f_n\|_{(p,\alpha)}
<\infty$ and $f_n \to 0$ uniformly on compact subset of $\CC$ as
$n\to \infty.$ Then following the same line of argument as in the
proof of the sufficiency of Theorem~\ref{bounded}, we obtain
\begin{eqnarray}
\|V_{(g,\psi)}(f_n)\|_{(q,\alpha)}^q  &\lesssim& \int_{\CC}
|f_n(w)|^q e^{-\frac{q\alpha}{2}|w|^2} B_{(\psi,\alpha)}(|g|^q)(w)
dm(w)\nonumber=I_n
  \label{end}
\end{eqnarray}
Then for a fixed $r>0,$ we split $I_n$ as
\begin{eqnarray}
I_n&=& \bigg(\int_{|w|\leq r} +\int_{|w|>r}\bigg) |f_n(w)|^q
e^{-\frac{q\alpha}{2}|w|^2} B_{(\psi,\alpha)}(|g|^q)(w) dm(w)\nonumber\\
&=& I_{n1} +I_{n2} \label{same}
\end{eqnarray} and estimate each piece independently. We first estimate  $I_{n1}.$
\begin{eqnarray*}
\limsup_{n \to \infty}I_{n1}&=&\limsup_{n \to \infty}\int_{|w|\leq
r} |f_n(w)|^q
e^{-\frac{q\alpha}{2}|w|^2} B_{(\psi,\alpha)}(|g|^q)(w) dm(w)\nonumber\\
&\leq& \limsup_{n \to \infty} \sup_{|w|\leq r}|f_n(w)|^q\int_{|w|
\leq r}e^{-\frac{q\alpha}{2}|w|^2} B_{(\psi,\alpha)}(|g|^q)(w) dm(w)\nonumber\\
&\lesssim& \limsup_{n \to \infty} \sup_{|w|\leq r}|f_n(w)|^q \to 0,
\ \text{ as} \ \ n\to \infty,
\end{eqnarray*} since $\sup_{w\in \CC }B_{(\psi,\alpha)}(|g|^q)(w)<\infty$. We
need to make  a similar conclusion for the second piece of integral
\begin{eqnarray*}
\limsup_{n \to \infty}I_{n2}&=& \limsup_{n \to \infty}\int_{|w|>r}
|f_n(w)|^q
e^{-\frac{q\alpha }{2}|w|^2} B_{(\psi,\alpha)}(|g|^q)(w) dm(w)\\
&\leq& \sup_{|w|>r}B_{(\psi,\alpha)}(|g|^q)(w) \limsup_{n \to
\infty}\|f_n\|_{(q,\alpha)}^q \\
&\leq& \sup_{|w|>r}B_{(\psi,\alpha)}(|g|^q)(w) \limsup_{n \to
\infty}\|f_n\|_{(p,\alpha)}^q
\end{eqnarray*}
Since $\sup_n \|f_n\|_{(p,\alpha)} <\infty$, we see that the last
expression in the right hand side above converges to zero when $r\to
\infty,$ and
hence $V_{(g,\psi)}(f_n) \to 0$ in $F_\alpha^q$ as $n\to \infty$.\\
 \emph{Proof of Theorem~\ref{thsmall}.}  Since (ii) obviously implies (i), we  shall show
that (iii)$\Rightarrow $(ii) and (i)$ \Rightarrow $ (iii). We first
assume that $B_{(\psi,\alpha)}(|g|^q)\in L^{p/(p-q)}(\CC, dm),$ and
show that $V_{(g,\psi)}$ is compact. Let $f_n$ be a sequence of
functions such that $\sup_n \|f_n\|_{(p,\alpha)}<\infty$ and $f_n$
converges to zero uniformly on compact subsets of $\CC.$ Then we
proceed as in the proof of
 Theorem~\ref{bounded} until we get equation \eqref{same}. We only
 need to modify our arguments in estimating the two piece of
 integrals $I_{n1}$ and $I_{n2}.$
Since $f_n \to 0$ uniformly on compact subsets of $\CC$
\begin{eqnarray*} \label{partly}\limsup_{n \to
\infty}I_{n1}&=& \limsup_{n \to \infty}\int_{|w|\leq r} |f_n(w)|^q
e^{-\frac{q\alpha}{2}|w|^2} B_{(\psi,\alpha)}(|g|^q)(w) dm(w)\nonumber\\
&\leq& \limsup_{n \to \infty} \sup_{|w|\leq
r}|f_n(w)|^q\int_{|w|\leq r} e^{-\frac{q\alpha}{2}|w|^2}
B_{(\psi,\alpha)}(|g|^q)(w)
dm(w)\nonumber\\
&\lesssim& \limsup_{n \to \infty} \sup_{|w|\leq r}|f_n(w)|^q \to 0 ,
\ n\to \infty.\end{eqnarray*} The last integral above converges
because \begin{equation*}\int_{|w|\leq r}
e^{-\frac{q\alpha}{2}|w|^2} B_{(\psi,\alpha)}(|g|^q)(w)
dm(w)\lesssim \bigg(\int_{|w|\leq r}B^{s}_{(\psi,\alpha)}(|g|^q)(w)
 dm(w)
\bigg)^{\frac{1}{s}}\end{equation*} by H$\ddot{\text{o}}$lder's
inequality where we set $s= p/(p-q)$ for brevity.  Again by
H$\ddot{\text{o}}$lder's inequality we obtain,
\begin{eqnarray*}
\limsup_{n \to \infty}I_{n2}&=& \limsup_{n \to \infty}\int_{|w|>r}
|f_n(w)|^q
e^{-\frac{q}{2}|w|^2} B_{(\psi,\alpha)}(|g|^q)(w) dm(w)\\
&\lesssim& \bigg(\int_{|w|>r}B^{s}_{(\psi,\alpha)}(|g|^q)(w)
 dm(w) \bigg)^{\frac{1}{s}} \limsup_{n
\to \infty}\|f_n\|_{(p,\alpha)}^q.
\end{eqnarray*}
 Since $\sup_n \|f_n\|_{(p,\alpha)}<\infty$,  we let $r \to \infty$ in
the above relation and with \eqref{partly} we  conclude that
$V_{(g,\psi)} \to 0$ as $ n\to \infty.$ Thus $ V_{(g,\psi)}$ is
compact. Obviously,  (i) follows from (ii). Thus our proof will be
complete once we show that (iii) follows from (i). To this end,  we
observe that  $V_{(g,\psi)}$ is bounded if and only if
\begin{eqnarray}
\label{carleson} \int_{\CC} |V_{(g,
\psi)}f(z)|^qe^{\frac{-q\alpha}{2}|z|^2}dm(z)&\simeq& \int_{\CC}
|f(z)|^q
d\mu_{(q,\alpha)}(z)\nonumber\\
&=& \int_{\CC}|f(z)|^q
e^{\frac{-q\alpha}{2}|z|^2}d\lambda_{(q,\alpha)}(z)\nonumber\\
&\lesssim&\|f\|_{(p,\alpha)}^q
\end{eqnarray} where $d\lambda_{(q,\alpha)}(z)= e^{\frac{q\alpha}{2}|z|^2}d\mu_{(q,\alpha)}(z)$.
The inequality in \eqref{carleson} means that $\lambda_{(q,\alpha)}$
is a $(p,q)$ Fock--Carleson measure. By Theorem~3.3 in \cite{ZHXL},
this holds if and only if
\begin{equation}
 \widetilde{\lambda_{(q,\alpha)}}(w)= \int_{\CC} |k_{(w,\alpha)}(z)|^q
e^{\frac{-q\alpha}{2}|z|^2}d\lambda_{(q,\alpha)}(z)\in
L^{p/(p-q)}(\CC, dm).\end{equation} Substituting back
$d\lambda_{(q,\alpha)}$ and $d\mu_{(q,\alpha)}$ in terms of $dm,$ we
obtain
\begin{eqnarray*}
\widetilde{\lambda_{(q,\alpha)}}(w)&=& \int_{\CC}
|k_{(w,\alpha)}(z)|^q
e^{\frac{-q\alpha}{2}|z|^2} d\lambda_{(q,\alpha)}(z)=\int_{\CC}|k_{(w,\alpha)}(z)|^q d\mu_{(q,\alpha)}(z)\\
&\simeq& \int_{\CC}|k_{(w,\alpha)}(z)|^q
\bigg|\frac{g'(z)}{1+|z|}\bigg|^q e^{\frac{-q\alpha}{2}|z|^2} dm(z)
= B_{(\psi,\alpha)}(|g|^q)(w)
\end{eqnarray*}
We remain to prove the norm estimate in \eqref{normless}. But this
can be easily seen as follows.  Since  $\lambda_{(q,\alpha)}$ is an
$(p, q)$ Fock--Carleson measure, the series of norm estimates in
 Theorem~3.3 in \cite{ZHXL} yields
\begin{equation*}
\|V_{g,\psi}\|\simeq
\Big(\|\widetilde{\lambda_{(q,\alpha)}}\|_{L^{p/(p-q)}(\CC,dm)}\Big)^{1/q}\simeq
\Big(\|B_{(\psi,\alpha)}(|g|^q)\|_{L^{p/(p-q)}(\CC,dm)}\Big)^{1/q}
\end{equation*}
which completes the proof of the theorem.

 Theorem~\ref{essentialnorm} follows from application of Lemmas
 \ref{lemdfinite}-\ref{pointwise}, Theorem~\ref{bounded}
 and appropriate  combination of arguments used to prove similar results
 in \cite{ZZH, UEKI1,UEKI, UEKI2}. Recall that each entire function $f$ can be expressed as $f(z)=
\sum_{k=0}^\infty p_k(z)$ where the $p_k's$ are  polynomials of
degree $k.$ We consider  a sequence of operators $R_n$  defined by
\begin{equation*}(R_n f)(z)= \sum_{k=n}^\infty
p_k(z).\end{equation*} It was proved in \cite{DGP,UEKI1} that
$\lim_{n\to \infty} \|R_nf\|_{(p,\alpha)} = 0$ for each $f$ in $
F_\alpha^p$, and hence $ \sup_n\|R_n\|<\infty.$ We need the
following more lemma in proving the theorem.
\begin{lemma}\label{generashipro}
 Let $1<p\leq q<\infty$ and $\psi$ be an entire function.
  If $V_{(g,\psi)}: F_\alpha^p \to  F_\alpha^q$ is bounded, then
 \begin{equation*}
\|V_{(g,\psi)}\|_e \leq \liminf_{n \to \infty} \| V_{(g,\psi)}
R_n\|_{(q,\alpha)}.
\end{equation*}
\end{lemma}

The proof of the lemma is  similar to the proof of Lemma~2 in  \cite{UEKI2}, and we omit it.\\
 \emph{Proof of Theorem~\ref{essentialnorm}.} We first prove the lower estimate in the
theorem.  We follow the ideas in the proofs of similar results for
weighted composition operators in \cite{ZZH,UEKI}. Let $Q$ be a
compact operator on $F_\alpha^p$. Since
$\|k_{(w,\alpha)}\|_{(p,\alpha)}=1$ and $k_{(w,\alpha)}$ converges
to zero uniformly on compact subset of $\CC$ as $|w| \to \infty$, we
have
\begin{eqnarray*}
\|V_{(g,\psi)}-Q\| &\geq& \limsup_{|w| \to \infty}
\|V_{(g,\psi)}k_{(w,\alpha)}-Q
k_{(w,\alpha)}\|_{(q,\alpha)} \\
&\geq& \limsup_{|w| \to
\infty}\|V_{(g,\psi)}k_{(w,\alpha)}\|_{(q,\alpha)}-\|Q
k_{(w,\alpha)}\|_{(q,\alpha)}\\
&=& \limsup_{|w| \to
 \infty}\|V_{(g,\psi)}k_{(w,\alpha)}\|_{(q,\alpha)}\\
 &\simeq& \bigg(\limsup_{|w| \to
 \infty} B_{(\psi,\alpha)} ( |g|^q)\bigg)^{1/q},
\end{eqnarray*}where the first equality is due to compactness of $Q.$
To prove the upper inequality, we follow the arguments in the proof
of Theorem~\ref{bounded}. For each unit norm $f$ in $F_\alpha^p$,
 we get
\begin{eqnarray*}
 \|V_{(g,\psi)}R_n f\|_{(q,\alpha)}^q&\simeq &  \int_{\CC} |R_nf(z)|^q d\mu_{(q,\alpha)}(z)\\
&\lesssim&  \int_{\CC}e^{\frac{q\alpha}{2}|z|^2}
d\mu_{(q,\alpha)}(z) \int_{\CC}\chi_{D(z,1)}(w) |R_nf(w)|^q e^{\frac{p\alpha}{2}|w|^2}dm(w)\\
&\lesssim& \int_{\CC} |R_nf(w)|^q e^{-\frac{q\alpha}{2}|w|^2}
B_{(\psi,\alpha)}(|g|^q)(w)dm(w)\\
&=& \bigg(\int_{ \CC\setminus D(0,r)}+
\int_{D(0,r)}\bigg)|R_nf(w)|^q e^{-\frac{q\alpha}{2}|w|^2}
B_{(\psi,\alpha)}(|g|^q)(w)dm(w)\\
&=& I_{n1}+ I_{n2}
\end{eqnarray*}
where for some fixed $r>0,$
\begin{eqnarray*}
I_{n1}&=& \int_{ \CC\setminus D(0,r)} |R_nf(w)|^q
e^{-\frac{q\alpha}{2}|w|^2} B_{(\psi,\alpha)}(|g|^q)(w)dm(w)\\
&& \ \ \ \ \ \ \lesssim \sup_{ \CC\setminus D(0,r)}
B_{(\psi,\alpha)}(|g|^q)(w)
\end{eqnarray*} which follows since $
\sup_n\|R_n\|<\infty,$ and \begin{equation*} I_{n2}= \int_{D(0,r)}
|R_nf(w)|^q e^{-\frac{q\alpha}{2}|w|^2}
B_{(\psi,\alpha)}(|g|^q)(w)dm(w). \end{equation*} We remain to
estimate $I_{n2}.$ By Lemma 1 in \cite{UEKI2}, we obtain,
\begin{equation}\label{strilling}
I_{n2} \lesssim\sup_{ w\in \CC} B_{(\psi,\alpha)}(|g|^q)(w) I_{n3}
\int_{\CC}e^{-\frac{q\alpha}{2}|w|^2} dm(w)
\end{equation} where
\begin{equation*}
 I_{n3}=\Bigg( \sum_{k= n}^\infty
\frac{|r|^k}{k!}\bigg( \big(2 s^{-1}\big)^{2^{-1}sk +1}\Gamma\big(
2^{-1}sk +1\big)\bigg)^{1/s}\Bigg)^q
\end{equation*} with $s$ the conjugate exponent of $p$ and $\Gamma$ is the Gamma
function.  By Stirling's formula, it holds
 \begin{equation*}
 \frac{|r|^k}{k!}\bigg( \big(2
s^{-1}\big)^{2^{-1}sk +1}\Gamma \big( 2^{-1}sk +1\big)\bigg)^{1/s}
\simeq \frac{|r|^k}{k!} \big(2 s^{-1}\big)^{2^{-1}k } \big( 2^{-1}sk
+1\big)^{\frac{k}{2}+\frac{1}{s}+\frac{1}{2s}} e^{-k/2}
 \end{equation*} when $k\to \infty.$
By ratio test,  the  series
\begin{equation*} \sum_{k=0}^\infty
\frac{|r|^k}{k!} \big(2 s^{-1}\big)^{2^{-1}k } \big( 2^{-1}sk
+1\big)^{\frac{k}{2}+\frac{1}{s}+\frac{1}{2s}}
e^{-k/2}\end{equation*} converges and hence $I_{n3}$ goes to zero
when $n\to \infty.$ By Theorem~\ref{bounded}, it follows that
$I_{n2} \to 0$ as $n\to \infty.$ Therefore
 \begin{equation*}
 \lim_{n\to \infty} \sup_{\|f\|_{(p,\alpha)}\leq 1}
\|V_{(g,\psi)}R_n f\|_{(q,\alpha)}^q\lesssim \sup_{ \CC\setminus
D(0,r)} B_{(\psi,\alpha)}(|g|^q)(w).\end{equation*} By Lemma
\ref{generashipro} we get \begin{equation*} \|V_{(g,\psi)}\|_e^q
\lesssim \lim_{r \to \infty}\sup_{ \CC\setminus D(0,r)}
B_{(\psi,\alpha)}(|g|^q)(w)\simeq \limsup_{|w| \to \infty}
B_{(\psi,\alpha)}(|g|^q)(w)
\end{equation*} and completes the proof.

 \emph{Proof of Theorem~\ref{Schatten}.}  The crucial step in proving the theorem
  is to introduce  a Teoplitz operator
on $F_\alpha^2.$ Let $\mu$ be a finite positive Borel measure on
$\CC$ satisfying the admissibility condition \begin{equation}
\label{admissible}\int_{\CC} |K_{(w,\alpha)}(z)|^2 e^{-\alpha|w|^2}
d\mu(w)<\infty\end{equation} for all $z\in \CC$. Then we define a
Teoplitz operator by
\begin{equation*}
T_{\mu}f(z)=\int_{\CC}K_{(w,\alpha)}(z) f(w) e^{-\alpha|w|^2}d\mu(w)
\end{equation*} for each $z\in \CC.$ Since the kernel functions are
dense in $F_\alpha^2,$ it follow by the admissibility condition and
H$\ddot{\text{o}}$lder's inequality that $T_\mu$ is well-defined. We
observe that by Lemma \ref{lemdfinite}, the inner product
\begin{equation}\label{below}
\langle f, h\rangle= f(0)\overline{h(0)} +\int_{\CC}
f'(z)\overline{h'(z)} \frac{e^{-\alpha|z|^2}}{(1+|z|)^2}dm(z)
\end{equation} defines a norm which is equivalent to
the usual  norm on $F_\alpha^2.$ We prefer to use this norm since
this alternative approach has the advantage that it permits us to
associate product of  Volterra type integral and composition
operators with Teoplitz operators easily. In deed, if $V_{(g,\psi)}$
is  a bounded operator on $F_\alpha^2, $ then we claim that
$V_{(g,\psi)}^* V_{(g,\psi)}= T_\mu $ where $T_\mu$ is the Teoplitz
operator induced by  the measure $d\mu= \phi o \psi^{-1}$
where\begin{equation*}
\phi(z)=|g'(z)|^2(1+|z|)^{-2}e^{-\alpha|z|^2}dm(z).
\end{equation*}
To show the claim, we consider a function  $f$  in $F_\alpha^2$ and
compute
\begin{eqnarray*}
 V^*_{(g,\psi)}V_{(g,\psi)} (f)(z)&=&\langle V^*_{(g,\psi)}V_{(g,\psi)} f,K_{(z,\alpha)} \rangle= \langle
V_{(g,\psi)} f, V_{(g,\psi)}K_{(z,\alpha)} \rangle\\
&=&  \int_{\CC}
f(\psi(w))\overline{K_{(z,\alpha)}(\psi(w))}|g'(w)|^2 (1+|w|)^{-2}e^{-\alpha|w|^2} dm(w)\\
&=& \int_{\CC} f(\psi(w))K_{(\psi(w),\alpha)}(z)|g'(w)|^2
(1+|w|)^{-2}e^{-\alpha|w|^2} dm(w)\\
 &=& \int_{\CC}f(\eta )K_{(\eta,\alpha)}(z)e^{-\alpha|\eta|^2}d\mu(\eta)=T_\mu(f)(z),
 \end{eqnarray*} follows from change of variables.   This shows that
 $ T_{\mu}= V^*_{(g,\psi)} V_{(g,\psi)}.$ For such particular  measure
 $\mu$, the admissibility condition \eqref{admissible} holds whenever $V_{(g,\psi)}$
 is bounded on $F_\alpha^2.$
Denote the associated Berezin symbol $\widetilde{\mu}$ of $\mu$ by
\begin{equation*}
 \widetilde{\mu}(z)= \langle T_\mu k_{(z,\alpha)},
k_{(z,\alpha)}\rangle, \ \ z \in \CC.
\end{equation*} Then an important result from  \cite{JIKZ} ensures that  the Toeplitz
operator $T_\mu$ belongs $S_p$ if and only if $\widetilde{\mu}$
belongs to $L^p(\CC, dm)$ for each $p>0$. On the other hand,
  $V_{(g,\psi)}$ belongs
to $S_p$ if and only if $V^*_{(g,\psi)} V_{(g,\psi)}$ belongs to
$S_{p/2}$ (see, \cite{KZ}), and this holds if and only if   $
\widetilde{\mu}(z)= \|V_{(g,\psi)}k_{(z,\alpha)}\|_{(2,\alpha)}^2$
belongs to $ L^{p/2}(\CC, dm)$. It is easily seen that
\begin{equation*}\|V_{(g,\psi)}k_{(z,\alpha})\|_{(2,\alpha)}^2\simeq
B_{(\psi,\alpha)}(|g|^2)(z)\end{equation*}
 and completes the proof.
 \subsection*{Acknowledgment}
 The author would like to thank professor
Kristian Seip for some constructive discussions on the subject
matter. \linespread{1.5}
 

\begin{thebibliography}{BRSHZ}
 \linespread{1.5}
\bibitem{Alman} A. Aleman, \emph{A class of integral operators on spaces of analytic functions}, Topics in complex
analysis and operator theory, 3?30, Univ. Malaga, Malaga, 2007.
 \bibitem{Alsi1} A. Aleman  and A. Siskakis, \emph{An integral operator on $H^p$}, Complex Variables, \textbf{28} (1995),
 149--158.
 \bibitem{Alsi2} A. Aleman  and A. Siskakis, \emph{Integration operators on Bergman
 spaces,} Indiana University Math J., \textbf{46} (1997), 337--356.
 \bibitem{CMS} B. Carswell, B. MacCluer,  and A. Schuster, \emph{Composition operators on the Fock
space}, Acta Sci. Math. (Szeged), \textbf{69} (2003), 871--887.
\bibitem{Olivia} O. Constantin, \emph{Volterra type integration operators on Fock
spaces,} Proceeding of American Mathematical Society, 140 (2012), no. 12,  4247--4257.
\bibitem{ZZHH} Z. Cuckovi$\acute{\text{c}}$ and R. Zhao, \emph{Weighted composition operators between
different weighted  Bergman spaces and different Hardy spaces,}
Illinois Journal of mathematics, \textbf{51} (2007), no. 2,
479--498.
\bibitem{ZZH} Z. Cuckovi$\acute{\text{c}}$ and R. Zhao, \emph{Weighted composition operators on
the Bergman space,} J. London Math. Soc., \textbf{70} (2004),
499--511.
\bibitem{DGP} D. Garling and P. Wojtaszczyk, \emph{Some Bergman spaces of analytic functions,}
Proceedings of the Conference on Function spaces, Edwardsville,
Lecture Notes in Pure and Applied Mathematics, \textbf{172} (1995),
123--138.
\bibitem{HOLDW} F. Holland and D. Walsh, \emph{Hankel operators in von
Neumann--Schatten classes}, Illinois J. Math.,  \textbf{32} (1988),
1--22.
\bibitem{ZHXL} Z. Hu and X. Lv, \emph{Teoplitz operators from one Fock space to
another}, Integr. Equ. Oper. Theory, \textbf{70} (2011),  541--559.
\bibitem{LIS} S.  Li and S. Stevi$\acute{\text{c}}$, \emph{Generalized composition operators on Zygmund
spaces and Bloch type spaces}, J. Math. Anal. Appl., \textbf{338}
(2008), 1282--1295.
\bibitem{JIKZ} J. Isralowitz and  K. Zhu, \emph{Toeplitz operators on the Fock
space,} Integral Equations Operator Theory, \textbf{66}  (2010), no.
4,  593--611.
\bibitem{SLI} S. Li, \emph{Volterra composition operators
between weighted Bergman space and Block type spaces,}
 J. Korean Math. Soc., \textbf{45} (2008), 229--248.
  \bibitem{PP} M. Pavlovic and  J. Pel$\acute{\text{a}}$, \emph{An equivalence for weigted integrals of an
analytic function and its derivative,} Math. Nachr., \textbf{281}
(2008), 1612--1623.

\bibitem{Pom} C. Pommerenke, \emph{Schlichte Funktionen
und analytische Funktionen von beschrenkter mittlerer Oszillation,}
Commentarii Mathematici Helvetici,  \textbf{52} (1977), no.4  ,
 591--602.
\bibitem{Shapiro} J. Shapiro, \emph{The essential norm of a composition operator,}
  Annals of Math., \textbf{125} (1987), 375--404.
\bibitem{AS} A. Sharma,   \emph{Volterra composition operators beteween weighted Bergman--Nevanlinna and
Bloch-type spaces}, Demonstratio  Mathemtica, Vol. XLII, \textbf{3}
(2009).
 \bibitem{Si} A. Siskakis, \emph{Volterra operators on spaces of analytic functions- a
survey,} Proceedings of the first advanced course in operator theory
and complex analysis, 51--68, Univ. Sevilla Secr. Publ., Seville,
2006.
\bibitem{UEKI1} S. I.  Ueki, \emph{Weighted composition operator on the
Bergman--Fock spaces,} Int. J. Mod. Math., \textbf{3} (2008),
231--243.
\bibitem{UEKI} S. I. Ueki, \emph{Weighted composition operator on the
Fock space,} Proc. Amer. Math. Soc., \textbf{135 (5)} (2007),
1405--1410.
\bibitem{UEKI2} S. I.  Ueki, \emph{Weighted composition operator on some
function spaces of entire functions,} Bull. Belg. Math. So. Simon
Stevin, \textbf{17} (2010), 343--353.
\bibitem{XZHU} X. Zhu, \emph{Generalized composition operators
and Volterra composition operators on Bloch spaces in the unit
ball,}  Complex Variables and Elliptic Equations: An International
Journal, \textbf{54}\textbf{ (2)} (2009), 95--102.
\bibitem{KZ} K. Zhu, \emph{Operator theory in function spaces,}
Mathematical Surveys and Monographs, Vol \textbf{138}. Amer. Math.
Soc., Providence, RI., 2007.
\bibitem{EW} E. Wolf, \emph{Volterra composition operators between weighted Bergman
spaces and weighted Bloch type spaces,} Collectanea Mathematica,
\textbf{61} (2010), no. 1, 57--63.
\bibitem{DV} D. Vukoti$\acute{\text{c}}$, \emph{Pointwise multiplication operators between Bergman
spaces on simply connected domains}, Indiana Univ. Math. J.,
\textbf{48} (1999), 793--803.
\end{thebibliography}
\end{document}